\documentclass[a4paper,12pt]{amsart}
\usepackage{amssymb}
\usepackage{mathrsfs}
\usepackage{color}

\theoremstyle{plain}
\newtheorem{theorem}{Theorem}[section]
\newtheorem{prop}[theorem]{Proposition}

\newtheorem{lemma}[theorem]{Lemma}

\theoremstyle{definition}
\newtheorem{remark}[theorem]{Remark}

\newcommand{\C}{\mathbb{C}}
\newcommand{\R}{\mathbb{R}}
\newcommand{\Z}{\mathbb{Z}}

\newcommand{\CP}{\mathbb{C}\mathrm{P}}

\renewcommand{\tilde}{\widetilde}
\renewcommand{\setminus}{\smallsetminus}
\newcommand{\nin}{/\kern-2.1ex\in}

\newcommand{\norm}[1]{\lVert#1\rVert}

\def\<{\left\langle}
\def\>{\right\rangle}

\def\ind{\operatorname{ind}}

\numberwithin{equation}{section}

\title[Geodesic flows and local Riemann-Roch numbers]{Geodesic flows on spheres and \\ the local Riemann-Roch numbers}

\author[H. Fujita]{Hajime Fujita$^1$}
\author[M. Furuta]{Mikio Furuta$^2$}
 \author[T. Yoshida]{Takahiko Yoshida$^3$}

\subjclass[2000]{} 
\thanks{$^1$Partly supported by Grant-in-Aid for Young Scientists (B) 23740059.}
\thanks{$^2$Partly supported by Grant-in-Aid for Scientific Research (A) 19204003 and Grant-in-Aid for Scientific Research (B) 19340015.}
\thanks{$^3$Partly supported by Grant-in-Aid for Young Scientists (B) 22740046 and Fujyukai Foundation.}

\address{Department of Mathematical and Physical Sciences, Japan Womens's University, 
2-8-1 Mejirodai, Bunkyo-ku, Tokyo, 112-8681, Japan}
\email{fujitah@fc.jwu.ac.jp}

\address{Graduate School of Mathematical Sciences, The University of Tokyo, 
3-8-1 Komaba, Meguro-ku, Tokyo, 153-8914, Japan}
\email{furuta@ms.u-tokyo.ac.jp}

\address{Department of Mathematics, Graduate School of Science and Technology, Meiji University, 1-1-1 Higashimita, Tama-ku, Kawasaki, 214-8571, Japan}
\email{takahiko@meiji.ac.jp}

\begin{document}

\maketitle

\begin{abstract}
We calculate the local Riemann-Roch numbers of the zero sections of $T^*S^n$ and $T^*\R P^n$,
where the local Riemann-Roch numbers are  defined by using 
the $S^1$-bundle structure on their complements associated to the geodesic flows. 
\end{abstract}


\section{Introduction}

In our previous papers \cite{Fujita-Furuta-Yoshida1}, 
\cite{Fujita-Furuta-Yoshida2}, and \cite{Fujita-Furuta-Yoshida3} we gave a formulation of index for  Dirac-type operators
on open manifolds when some additional structures
are given on the ends of the base manifolds.  A typical example is
the {\it local Riemann-Roch number}  for an open neighborhood of
a singular fiber of (not necessarily completely) integrable system.
In our previous papers we mainly considered the cases of global torus actions.


The explicit examples we calculated  in  \cite{Fujita-Furuta-Yoshida1},
\cite{Fujita-Furuta-Yoshida2}, and \cite{Fujita-Furuta-Yoshida3}  were
2-dimensional cases and their products,
which are integrable systems having tori as singular fibers.
In this paper we determine
the local Riemann-Roch number
of $T^* S^n$ 
when the additional structure is
given by the geodesic flow for the 
standard metric. The two new aspects
of this example is as follows:
(1) The case 
$T^*S^n$ is a Hamiltonian system
having the zero section $S^n$ as singularities. 
(2) On the complement of $S^n$
we have the $S^1$-action induced from the geodesic flow,
while the $S^1$-action cannot be extended to
the whole spaces.


The examples we consider in this paper
may appear as parts of
the completely integrable systems
on the moduli spaces of $SU(2)$ or $SO(3)$ flat connections
on Riemann surfaces constructed in \cite{Jeffrey-Weitsman} \cite{Weitsman} .
A possible application of our framework is to give a direct proof of
the Verlinde formula for this case
based on the properties of our local Riemann-Roch numbers
developed in \cite{Fujita-Furuta-Yoshida2}.
We expect the calculation in this paper will be a key statement  in this application.

%
The case of $T^*S^n$ is the double covering
of the case of $T^*\R P^n$. However their local Riemann-Roch numbers
are not related by the multiplication by $2$.
This observation implies that the local Riemann-Roch number
is a global invariant which is not expressed as integral of some local invariant.

We use the excision formula and the product formula to
 calculate the local Riemann-Roch number of $T^*S^n$.
 We first construct a compactification with a compatible almost complex structure.
 The Riemann-Roch number of the compactification is directly calculated
 by the usual Riemann-Roch formula. What we need is to identify 
 the contribution of the zero section $S^n$.
 The other contributions to the whole Riemann-Roch number
 is calculated by the product formula.
 We obtain the required number by the excision formula and
 a simple subtraction. 


The organization of the present paper is as follows.
In Section~2 we state
a formulation of local (or relative) index in \cite{Fujita-Furuta-Yoshida1}
which will be convenient to use in this paper.
In Section~3 we state our main theorem.
In Section~4 we construct the three compactifications of
$T^*S^n$ with almost complex structures.
In Section~4.4 we compare these three and
show that we can use any of them to calculate
the required local Riemann-Roch numbers.
In Section~5 we show general properties
of local models, in particular, a product formula
for the product of discs or cylinders 
and closed manifolds, which is applied to 
the case of symplectic cuts.
In Section~6 we prove our main theorem.

\subsection{Notations}
Let 
$E\to X$ be a vector bundle with Euclidean metric. 
For a positive number $r$ 
we denote by $D_r(E)$, $D_r^o(E)$ and $S_r(E)$ 
the closed disc bundle, the open disc bundle  
and the circle bundle of radius $r$, respectively.

\section{Circle actions and local Riemann-Roch number}
In this section we recall our setting in \cite{Fujita-Furuta-Yoshida1}, 
which allows us to define the local Riemann-Roch number and 
to have a localization formula of the Riemann-Roch number.  
For our convenience we explain in the category of almost Hermitian manifolds 
with $S^1$-actions. 
See \cite{Fujita-Furuta-Yoshida1} for the general version. 

Let $(M,J,g)$ be an almost Hermitian manifold, i.e., 
$M$ is a smooth manifold, $J$ is an almost complex structure on $M$ and 
$g$ is a Riemannian metric of $M$ which is preserved by $J$. 
Let $(E,\nabla)\to M$ be a Hermitian vector bundle with connection. 
Let $W_E$ be a $\Z/2$-graded Hermitian vector bundle 
$\wedge^{\bullet}TM\otimes E$, which has a structure of 
Clifford module bundle over $TM$.  
Suppose that there exists an open subset $V$ of $M$ 
with an action of the circle group $S^1$ 
which satisfies the following conditions. 
\begin{enumerate}
\item The complement $M\setminus V$ is compact. 
\item The $S^1$-action preserves the 
almost Hermitian structure $(J,g)$ on $V$. 
\item There are no fixed points, i.e., $V^{S^1}=\emptyset$. 
\item The restriction of $(E,\nabla)$ to any $S^1$-orbit 
does not have non-zero parallel sections. 
\end{enumerate}

In \cite{Fujita-Furuta-Yoshida1} we defined the 
{\it relative (or local) index} ${\rm ind}(M,V,W_E)={\rm ind}(M,V)$ 
which satisfies the following. 

\begin{theorem}
The relative index ${\rm ind}(M,V)$ satisfies the following properties. 
\begin{enumerate}
\item If $M$ is closed, then ${\rm ind}(M,V)$ 
coincides with the index of the spin$^c$ Dirac operator 
acting on the space of sections of $W_E$. 
\item ${\rm ind}(M,V)$ is invariant under continuous deformations of the data, i.e., 
$(J,g,E,\nabla)$ and the $S^1$-action on $V$. 
\item ${\rm ind}(M,V)$ satisfies an excision formula. 
\item ${\rm ind}(M,V)$ satisfies a product formula. 
\end{enumerate}
\end{theorem}
The relative index does not depend on the choice of the neighborhood of 
the complement $K:=M\setminus V$. 
Namely when $M'$ is a neighborhood of $K$, if we put $V':=M'\setminus K$,  
we have $\ind(M,V)=\ind(M',V')$ from the excision formula (3). 
Although the relative index depends on the data on 
a neighborhood of $K$ we write $\ind_{\rm loc}(K)$ for this index 
if there are no confusion. 
If $M$ is a symplectic manifold with a compatible almost complex structure 
and $(E,\nabla)$ is a prequantizing line bundle, 
then we denote it by ${\rm ind}(M,V)=RR(M,V)$ or 
$\ind_{\rm loc}(K)=RR_{\rm loc}(K)$ and 
call it the {\it local (or relative) Riemann-Roch number}.


\section{Main theorem}
Let $S^n$ be the unit sphere centered at the origin in the $(n+1)$-dimensional Euclidean space $\R^{n+1}$. 
We denote by $V_0$ the total space of the cotangent bundle of $S^n$ and 
by $V_{01}$ the complement of the zero section in $V_0$. 
We identify the cotangent bundle $V_0$ with the tangent bundle $TS^n$ 
via the Riemannian metric. 
Hence $V_0$ is identified with the set of pairs of vectors 
$(x,v)$ in $\R^{n+1}$ with $||x||=1$ and $x\cdot v=0$, 
where $x\cdot v$ is the standard Euclidean inner product and 
$||\cdot||$ is the induced norm. 


Consider the Liouville 1-form $\alpha$ and 
the standard symplectic structure $d\alpha$ of $V_0$. 
The Hamiltonian action for the Hamiltonian function 
$$
h_{01}:V_{01}\to \R, \ h_{01}(x,v):=||v||
$$is periodic with the period $2\pi$, and hence, 
induces a Hamiltonian (free) $S^1$-action on $V_{01}$. 
For $(x,v)\in V_{01}$ let $\R\langle x,v\rangle$ be the oriented 2-plane 
generated by $x,v\in \R^{n+1}$. 
The $S^1$-action on $V_{01}$ is given by the standard rotation action of 
$S^1=SO(2)$ on $\R\langle x,v\rangle$, which we call the normalized geodesic flow\footnote{It is well-known that the Hamiltonian vector field for the Hamiltonian function  $h_0(x,v):=\frac{1}{2}||v||^2$ gives the geodesic flow on $V_0$. Though this flow is periodic, it does not give an $S^1$-action in general. We need an $S^1$-action, so we use $h_{01}$ insead of $h_0$. 
}.

Consider the trivial bundle $V_0\times \C$ with the connection 1-form $\alpha$, which gives a prequantizing line bundle over $V_0$. 
Using the above $S^1$-action 
we can define the relative Riemann-Roch number $RR_{\rm loc}(S^n)$ 
of the zero section $S^n\subset V_0$, 
where we use the $S^1$-invariant $d\alpha$-compatible 
almost complex structure $J_0$, 
which is defined by the Levi-Civita connection of $S^n$ as in Subsection~4.4. 
The following is the main theorem in the present paper. 
\begin{theorem}\label{RRofSn}
$$
RR_{\rm loc}(S^n)=\left\{
\begin{array}{ll}
2 \quad (n=0) \\ 
1 \quad (n\geq 1). 
\end{array}\right.
$$
\end{theorem}
Since the integral of $\alpha$ along the $S^1$-orbit of $(x,v)\in V_{01}$ 
is given by $2\pi||v||$, the holonomy along the orbit is equal to 
$\exp(2\pi\sqrt{-1}||v||)$.  
In particular the following holds. 

\begin{lemma}\label{BSorbits}
The holonomy along the $S^1$-orbit of $(x,v)\in V_{01}$ 
is trivial if and only if 
the norm of $v$ is an integer. 
\end{lemma}
Let $l$ be a positive integer. 
By the above lemma for any small neighborhood $M_l$ of $h_{01}^{-1}(l)$ 
we can define the relative Riemann-Roch number 
$RR(M_l,M_l\setminus h_{01}^{-1}(l))=RR_{\rm loc}(h_{01}^{-1}(l))$. 
We will also show the following. 
\begin{theorem}\label{RRofBS}
$$
RR_{\rm loc}(h_{01}^{-1}(l))=
\left(\begin{array}{cc}
n+l-1 \\ n-1
\end{array}\right) + 
\left(\begin{array}{cc}
n+l-2 \\ n-1
\end{array}\right). 
$$
\end{theorem}

\begin{remark}
(1) The isometric action of $O(n+1)$ on $S^n$ extends to all the constructions in the subsequent sections. 
Using the corresponding  equivariant version of the localization, we can 
show that the above identification of the local Riemann-Roch number in Theorem~\ref{RRofSn} still  holds as virtual  $O(n+1)$ representation, i.e.,
the equivariant local Riemann-Roch character is equal to $1$ for $n>0$.
It implies that for any finite subgroup $G$ in $O(n+1)$ acting freely on $S^n$,
the local Riemann-Roch number of $S^n/G$ in $T(S^n/G)$ is equal to $1$ for $n>0$. \\
(2) In the setting of (1), for any nontrivial representation $\rho\colon G\to \{\pm 1\}$ we can twist the prequantizing line bundle by $\rho$ to obtain another prequantizing line bundle on the quotient space $T(S^n/G)$. The local Riemann-Roch number $S^n/G$ in $T(S^n/G)$ for this twisted prequantizing line bundle is similarly calculated and turns out to be $0$ for $n>0$. In particular, by taking $\{\pm 1\}$ as $G$ and taking the identity map as $\rho$ we can show that the local Riemann-Roch number of $\R P^n=S^n/G$ in $T\R P^n=TS^n/G$ with the twisted prequantizing line bundle is equal to $0$ for $n>0$. \\
%
(3) The above two examples imply that the local Riemann-Roch number
is not always multiplicative with respect to finite covering.
\end{remark}

%
%

\section{Compactification of the cotangent bundle}
\subsection{Symplectic cut}
For the prequantized symplectic manifold 
$(V_0,d\alpha,V_0\times\C,\alpha)$ and the 
Hamiltonian $S^1$-action on $V_{01}$ we can apply the construction of 
{\it symplectic cut} due to E.~Lerman\cite{Lerman}. 
Namely for $r>0$ let $\hat h_r:V_{01}\times \C\to \R$ 
be the function defined by 
$$
\hat h_r(x,v,w):=||v||+\frac{1}{2}|w|^2-r. 
$$
$\hat h_r$ is a Hamiltonian for the diagonal 
$S^1$-action on $V_{01}\times \C$ 
and $0$ is a regular value of $\hat h_r$. 
We define the symplectomorphism  $\varphi\colon D^o_r(T^*S^n)\cap V_{01}\to \{ (x,v,w)\in \hat h_r^{-1}(0)\colon w\neq 0\}/S^1$ by
\[
\varphi(x,v)=\left[x,v,\sqrt{2(r-\norm{v})}\right]. 
\]
Then the symplectic cut $X_r$ is defined by gluing $D^o_r(T^*S^n)$ and $\hat h_r^{-1}(0)/S^1$ by $\varphi$, i.e.,  
$$
X_r:=D_r^o(T^*S^n)\cup_\varphi \hat h_r^{-1}(0)/S^1.
$$
By definition $X_r$ can be identified with a union $D^o_r(T^*S^n)\cup h_{01}^{-1}(r)/S^1$ as a set.  
Note that $X_r$ is a smooth symplectic manifold, and 
if $r$ is a positive integer, then there exists the 
induced prequantizing line bundle $L_r$ over $X_r$. 
We put $M_{(r)}:=h^{-1}_{01}(r)/S^1$, which is a symplectic submanifold 
of $X_r$. The normal bundle $\nu_{(r)}$ of $M_{(r)}$ in $X_r$ is isomorphic to 
$h^{-1}_{01}(r)\times_{S^1}\C_1$, where $\C_1$ is the one-dimensional 
representation of $S^1$ of weight 1. 

\begin{lemma}\label{embednu}
The following gives an embedding of $\nu_{(r)}$ into $X_r$ as a 
tubular neighborhood of $M_{(r)}$: 
$$
\nu_{(r)}=h^{-1}_{01}(r)\times_{S^1}\C_1 \ni [ x,v,z]\mapsto 
\left[x,\frac{v}{1+|z|^2}, \frac{\sqrt{2r}z}{\sqrt{1+|z|^2}}\right]\in \hat h_r^{-1}(0)/S^1\subset X_r.  
$$
\end{lemma}

\begin{remark}
Though the smooth manifolds $X_r$ for different $r$ are 
all diffeomorphic to each other, 
$(X_r, L_r)$ are not mutually isomorphic as prequantized symplectic manifolds. 
\end{remark}
\subsection{Grassmannian}
For three nonnegative integers $a_0, a_1$ and $N$ 
with $a_0+a_1+1=N$, two oriented Grassmannians 
${\rm Gr}_{a_0+1}^+(\R^{N+1})$ and ${\rm Gr}_{a_1+1}^+(\R^{N+1})$ 
can be identified by taking orthogonal complements. 
We put $X_{a_0,a_1,N}={\rm Gr}_{a_0+1}^+(\R^{N+1})={\rm Gr}_{a_1+1}^+(\R^{N+1})$. 
For $i=0,1$, let $\gamma_{a_i}$ be the total space of 
the tautological $\R^{a_i}$-bundle 
over ${\rm Gr}_{a_i}^+(\R^{N})$ and $\gamma_{a_i}^{\perp}$ the 
orthogonal complement bundle of $\gamma_{a_i}$ in 
${\rm Gr}_{a_i}^+(\R^N)\times \R^N$. 

\begin{lemma}
$X_{a_0,a_1,N}$ is diffeomorphic to the manifold 
$\gamma_{a_0}^{\perp}\sqcup \gamma_{a_1}^{\perp}/\sim$, where 
$(P_0,v_0)\sim (P_1,v_1)$ for $(P_i,v_i)\in \gamma_{a_i}^{\perp}$ 
if $v_0,v_1\neq 0$ and the following three relations 
are satisfied: 
$$
P_0\perp P_1, \quad \frac{v_0}{||v_0||}+\frac{v_1}{||v_1||}=0, \quad  
||v_0||\cdot ||v_1||=1. 
$$
\end{lemma}
\begin{proof}
For $(P_i,v_i)\in \gamma_{a_i}^{\perp}$ let $P$ be the 
subspace $(P_i\oplus\{0\})\oplus\R(v_i\oplus 1)$ of $\R^{N+1}$. 
The map $\gamma_{a_i}^{\perp}\to X_{a_0,a_1,N}$, $(P_i,v_i)\mapsto P$ is compatible with the identification $\sim$ and induces the map $\gamma_{a_0}^{\perp}\sqcup \gamma_{a_1}^{\perp}/\sim\to X_{a_0,a_1,N}$.  

We next construct the inverse map $X_{a_0,a_1,N}\to \gamma_{a_0}^{\perp}\sqcup \gamma_{a_1}^{\perp}/\sim$ 
as follows. 
Suppose that $P\in X_{a_0,a_1,N}={\rm Gr}_{a_0+1}^+(\R^{N+1})$ 
is not contained 
in the subspace $\R^N\oplus\{0\}\subset \R^{N+1}$. 
We put $P_0:=P\cap (\R^N\oplus\{0\})$ and 
let $P_0^{\perp}$ be the orthogonal complement of $P_0$ in $\R^{N+1}$.  
Since $P$ is not contained in $\R^N$, 
the orthogonal projection $\R^{N+1}\to 
\{0\}\oplus\R\subset\R^{N+1}$ gives an 
isomorphism between $P\cap P_0^{\perp}$ and $\{0\}\oplus \R$. 
Define $v_0\in P_0^{\perp}\cap\R^N$ so that $v_0\oplus 1$ gives 
an oriented basis of the line $P\cap P_0^{\perp}$. 
By definition $(P_0,v_0)$ is an element in $\gamma_{a_0}^{\perp}$. 
If $P^\perp$ is not contained in $\R^N$, 
then we can define 
$(P_1,v_1)\in \gamma_{a_1}^{\perp}$ as in the above way 
for $P^{\perp}$. 
The map $P\mapsto (P_i,v_i)$ is compatible with the identification $\sim$ and induces the inverse map.  
%
\end{proof}

Hereafter we put $a_0=1$, $a_1=n-1$, $N=n+1$ and $X_{1,n-1,n+1}=X$. 
In this case we have 
${\rm Gr}_{a_0}^+(\R^N)=S^n$, $\gamma_{a_0}^{\perp}=TS^n=V_0$ and 
${\rm Gr}_{a_1}^+(\R^{N})={\rm Gr}_{n-1}^+(\R^{n+1})={\rm Gr}_{2}^+(\R^{n+1})$. Let $\gamma_2$ be the tautological plane bundle over ${\rm Gr}_{2}^+(\R^{n+1}). $
In this setting the identification $\sim$ can be described as follows: 
Let $r$ be a positive real number. 
For $(x,v)\in S_r(TS^n)$ we define $(P_1,v_1)\in S_{1/r}(\gamma_2)$ by 
$P_1^\perp:=\R\langle x,v\rangle$ and $v_1:=-v/r^2$. 
For $(P_1, v_1)\in S_{1/r}(\gamma_2)$ we put $v:=r^2v_1$
and define $x\in S^n$ so that $x$ is a tangent vector 
of the circle $P_1^\perp\cap S^n$ at $\frac{v}{||v||}$ and 
$\left\{x,\frac{v}{||v||}\right\}$ is an 
oriented orthonormal basis of $P_1^\perp$. 
These two maps give a diffeomorphism between 
$S_r(TS^n)$ and $S_{1/r}(\gamma_2)$ which induces the identification 
$$
X=TS^n\cup \gamma_2/\sim. 
$$
Under the above diffeomorphism $S_r(TS^n)\cong S_{1/r}(\gamma_2)$ 
the circle of radius $r$ in $\R\<x,v\>$ 
corresponds to the circle of radius $1/r$ in 
$P_1^\perp\in {\rm Gr}_2^+(\R^{n+1})$, and the $S^1$-action induced by 
the normalized geodesic flow is the principal $S^1$-action on $S_{1/r}(\gamma_2)$. 
It implies ${\rm Gr}_{2}^+(\R^{n+1})=h_{01}^{-1}(r)/S^1$ for all $r>0$. 
Moreover we have the following. 
\begin{lemma}\label{X_r=X}
$X_r$ is diffeomorphic to $X$ for any $r>0$. 
\end{lemma}

\begin{proof}
We first take and fix a positive number $\varepsilon>0$ small enough so that $\varepsilon<\frac{r}{2}$ and 
$(r-\varepsilon)\varepsilon<1$. 
Let $\rho:[0,r)\to \R$ be a smooth function such that 
$\rho(s)=1 \ (0\le s\le r-2\varepsilon)$, 
$\rho(s)=\frac{1}{\sqrt{s(r-s)}} \ (r-\varepsilon\le s <r)$ and $\rho$ is strictly 
increasing on $[r-2\varepsilon, r)$. 
We define a diffeomorphism $D_r^{o}(TS^n)\to TS^n$ by $(x,v)\mapsto (x,\rho(\norm{v})v)$. On the other hand there is a diffeomorphism between 
$\nu_{(r)}$ and $\gamma_2$ defined by 
$$
\nu_{(r)}\ni [ x,v,z] \mapsto 
\left(\R\langle x,v\rangle, -\left({\rm Im}(z)x+{\rm Re}(z)\frac{v}{\norm{v}}\right) \right)\in \gamma_2. 
$$
When we use the embedding of $\nu_{(r)}$ in Lemma~\ref{embednu} and 
descriptions $X_{r}=D_r^o(TS^n)\cup 
D_{r_{\varepsilon}}^o(\nu_{(r)})$ and $X=TS^n\cup 
D_{r_{\varepsilon}}^o(\gamma_2)$ for 
$r_{\varepsilon}=\sqrt{\frac{\varepsilon}{r-\varepsilon}}$,  
two maps defined above give a diffeomorphism $X_r\cong X$. 
\end{proof}

\begin{remark}\label{normalbundle}
By the above description the normal bundle of 
${\rm Gr}_{2}^+(\R^{n+1})$ (resp.~$S^n$) in $X=X_r$ is isomorphic to 
$\gamma_2$ (resp.~$TS^n$). 
\end{remark}

\subsection{Quadratic hypersurface} 

Let $Q_n$ be the quadratic hypersurface in 
$\CP^{n+1}$ defined by 
$$
Q_n:=\left\{[z_0:z_1:\cdots:z_{n+1}] \ \Bigm| \ \sum_{i=0}^{n+1}z_i^2=0\right\}. 
$$


\begin{lemma}\label{X_r=Q_n}
$X_r$ is diffeomorphic to $Q_n$ for any $r>0$. 
\end{lemma}
\begin{proof}
From Lemma~\ref{X_r=X} we have the diffeomorphism $X_r\cong {\rm Gr}_{2}^+(\R^{n+2})$. 
Let $P$ be an oriented 2-plane in $\R^{n+2}$ with an oriented 
orthonormal 2-frame $\{u_0,u_1\}$. 
The point $[u_0+\sqrt{-1}u_1]$ in $\CP^{n+1}$ is well-defined and 
contained in $Q_n$. 
The map $X_r\to Q_n$ gives the required diffeomorphism. 
\end{proof}

\begin{remark}\label{normalbundle2}
Under the above identification the divisor $Q_{n-1}\subset Q_n$ defined by 
$z_{n+1}=0$ corresponds to the submanifold 
${\rm Gr}_{2}^+(\R^{n+1})\subset
{\rm Gr}_2^+(\R^{n+2})$. 
The normal bundle of $Q_{n-1}$ in $Q_n$ is 
the tautological bundle ${\mathcal O}(1)|_{Q_{n-1}}$, 
which is isomorphic to the tautological bundle 
$\gamma_2\to {\rm Gr}_2^+(\R^{n+1})$ under the identification 
$Q_{n-1}={\rm Gr}_{2}^+(\R^{n+1})$. 
\end{remark}

\subsection{Comparison of almost complex structures}
\subsubsection{$J_0$, $J_1$ and $J_2$}
We first define three almost complex structures $J_0$, $J_1$ and $J_2$. 

\medskip

\paragraph{\underline{(0) $J_0$}}
Let $J_0$ be the standard almost complex structure on $V_0=T^*S^n=TS^n$ 
determined by 
the Riemannian metric and the Levi-Civita connection of $S^n$. 
Let $p:TS^n\to S^n$ be the projection and $\iota_{0}:p^*TS^n\to T(TS^n)$ be 
the splitting of the bundle map $T(TS^n)\to p^*TS^n$ determined by the 
Levi-Civita connection. 
Let $\iota_1:p^*TS^n\to T(TS^n)$ be the natural embedding of $p^*TS^n$ 
into the tangent bundle along fibers of $p:TS^n\to S^n$.  
We have a decomposition $T(TS^n)\cong p^*TS^n\oplus p^*TS^n$ 
via $\iota_0$ and $\iota_1$. The standard almost complex structure $J_0$ is 
characterized by the condition $J_0\iota_0=\iota_1$. 
Note that $J_0$ is invariant under the $S^1$-action of 
the normalized geodesic flow on $V_{01}$. 

Now we give more explicit description of $J_0$ 
using the natural embedding $TS^n\subset S^n\times \R^{n+1}$. 
We fix a non-zero tangent vector $(x,v)\in TS^n$.
Under this embedding one has 
$$
T_{(x,v)}(TS^n)=\{(u_1,u_2)\in \R^{n+1}\times \R^{n+1} \ | \ 
u_1\cdot x=u_2\cdot x+u_1\cdot v=0\}.  
$$
We give the description of the horizontal lift $T_xS^n\to T_{(x,v)}(TS^n)$ 
using the above description. 
For each $w\in T_xS^n$ the subspace $\R\< x,v,w\>$ generated by $x$,$v$ and $w$ 
is a 2-plane if $w$ is parallel to $v$ or a 3-space otherwise. 
Since the unit sphere in $\R\<x,v,w\>$ is totally geodesic in $S^n$, 
it suffices to consider the case $n=2$. 
We fix an oriented orthonormal frame $\{e_1, e_2, e_3\}$ of $\R^3$ 
so that $x=e_1$, $v=\norm{v}e_2$. 
Note that $e_2$ and $e_3$ form an oriented orthonormal basis of $T_xS^2$.   

\begin{lemma}\label{horlift}
The horizontal lifts of $e_2, e_3\in T_xS^2$ are given by 
$(e_2, -\norm{v}e_1), (e_3, 0)\in T_xS^2\times \R^3$ respectively. 
\end{lemma}

The natural embedding $T_xS^2\to T_{(x,v)}(TS^2)$ 
of the fiber direction is given by 
$T_xS^2\ni w\mapsto (0,w)\in T_xS^2\times \R^{3}$. 
By definition of $J_0$ we have the following characterization. 

\begin{lemma}\label{J_0}
$J_0$ is characterized by the condition 
$$
J_0 : (e_2, -\norm{v}e_1)\mapsto (0, e_2), \ (e_3, 0)\mapsto (0, e_3). 
$$
\end{lemma}

\vspace{0.5cm}

\paragraph{\underline{(1) $J_1$}}
We define an almost complex structure $J_1$ on 
$X={\rm Gr}_2^+(\R^{n+2})$ as follows. 
For each oriented 2-plane $P\in X$, 
its orientation and Euclidean metric determine 
a canonical complex structure 
$J_P$ on $P$. We define an almost complex structure 
$(J_1)_P$ on $T_PX={\rm Hom}(P,P^{\perp})$  by 
$$
(J_1)_P(f):=f\circ J_P^{-1}. 
$$

\vspace{0.5cm}

\paragraph{\underline{(2) $J_2$}}
Let $J_2$ be the standard complex structure on $Q_n$ 
as a projective variety in $\CP^{n+1}$.

\subsubsection{$J_0\sim J_1=J_2$ on $TS^n$}
Let $V_1$ be the complement of $S^n$ in $X={\rm Gr}^+_2(\R^{n+2})$. 
We give two types of decompositions of $TX|_{V_0\cap V_1}$. 

For $P\in V_1\subset {\rm Gr}_2^+(\R^{n+2})$ let $P_2$ be the image of $P$ 
under the natural projection $\R^{n+2}\to \R^{n+1}$, 
which is a 2-plane in $\R^{n+1}$. 
Let $P^{\perp}$ be the orthogonal complement of $P$ in $\R^{n+2}$ and 
$\check{P_2}$ the orthogonal complement of $P_2$ in $\R^{n+1}$. 
Note that  $\check{P_2}=\check{P_2}\oplus \{0\}$ is a codimension 1 subspace of $P^{\perp}$. 
Let $F_P$ be the orthogonal complement of $\check{P_2}$ in $P^{\perp}$. 
We put $E_P:={\rm Hom}(P,\check{P_2})$ and $L_P:={\rm Hom}(P, F_P)$. 
In this way we have the first decomposition 
$TX|_{V_1}=E\oplus L$. 
Note that $E$ and $L$ are $J_1$-invariant subbundles. 

For $(x,v)\in TS^n$ consider the subspace $\R\<x,v\>$ in $\R^{n+1}$. 
Let $E'_{(x,v)}$ be the subspace of $T_{(x,v)}(TS^n)$ 
defined by $E'_{(x,v)}:=\R\<x,v\>^{\perp}\times \R\<x,v\>^{\perp}$ in $\R^{n+1}\times \R^{n+1}$. 
Let $L'_{(x,v)}$ be the subspace of $T_{(x,v)}(TS^n)$ 
defined by $L'_{(x,v)}=(\R\<x,v\>\times \R\<x,v\>)\cap T_{(x,v)}(TS^n)$. 
From Lemma~\ref{J_0} note that if $v=0$, then $L_{(x,0)}'=\{0\}$ and 
if $v\neq 0$, then $L'_{(x,v)}$ has a natural basis 
$\left\{\left(\frac{v}{\norm{v}},-\norm{v}x\right), \left(0, \frac{v}{\norm{v}}\right)\right\}$. 
In this way we have the second decomposition 
$TX|_{V_0\cap V_1}=E'\oplus L'$. 
Note that $E'$ and $L'$ are $J_0$-invariant subbundles. 

\begin{lemma}
Under the identification 
$TS^n\setminus S^n=V_0\cap V_1={\rm Gr}_2^+(\R^{n+2})\setminus 
(S^n\cup {\rm Gr}_2^+(\R^{n+1}))$, 
the subbundle $E$ (resp. $L$) is identified with $E'$ (resp. $L'$). 
\end{lemma}

\begin{proof}
It suffices to prove for the case $n=2$.  
For $(x,v)\in TS^n\setminus S^n$ 
we use an oriented orthonormal frame $\{e_1, e_2,e_3\}$ of $\R^3$ as in 
Lemma~\ref{horlift} and Lemma~\ref{J_0}. 
In this case we have $E'_{(x,v)}=\R\<(e_3,0), (0,e_3)\>$ 
and $L'_{(x,v)}=\R\<(e_2, -\norm{v}e_1),(0, e_2)\>$.  
We put $P=\R\<x\oplus 0, v\oplus 1\>\in {\rm Gr}_2^+(\R^{n+2})$, and then 
we have a natural basis $e_2\oplus (-\norm{v})$ of $F_P$. 
Under the identification $T_{(x,v)}(TS^n)=T_P{\rm Gr}_2^+(\R^{n+2})$ 
one can check the following correspondence by direct computations. 
$$
(e_3,0)\mapsto f_{(e_3, 0)}, \ (0, e_3)\mapsto f_{(0, e_3)}, \ 
(e_2, -\norm{v}e_1)\mapsto f_{(e_2, -\norm{v}e_1)}, \ 
(0,e_2)\mapsto f_{(0,e_2)}, 
$$where linear maps $f_{(e_3,0)}\in T_P{\rm Gr}_2^+(\R^{n+2})$ etc. 
are defined by 
$$
f_{(e_3,0)}: 
\left\{\begin{array}{ll} 
x\oplus 0 \mapsto e_3\oplus 0 \\ 
v\oplus 1\mapsto 0,  
\end{array}\right.
f_{(0, e_3)}: 
\left\{\begin{array}{ll} 
x\oplus 0 \mapsto 0  \\ 
v\oplus 1\mapsto  e_3\oplus 0, 
\end{array}\right. 
$$
$$
f_{(e_2, -\norm{v}e_1)}: 
\left\{\begin{array}{ll} 
x\oplus 0 \mapsto \frac{1}{1+\norm{v}^2}(e_2\oplus(-\norm{v})) \\ 
v\oplus 1\mapsto 0,  
\end{array}\right.
f_{(0, e_2)}: 
\left\{\begin{array}{ll} 
x\oplus 0 \mapsto 0  \\ 
v\oplus 1\mapsto  \frac{1}{1+\norm{v}^2}(e_2\oplus(-\norm{v})). 
\end{array}\right.
$$
It implies that $E'_{(x,v)}$ (resp. $L'_{(x,v)}$) is 
isomorphic to $E_P$ (resp. $L_P$). 
\end{proof}

\begin{remark}
Though $E$ (resp. $L$) and $E'$ (resp. $L'$) 
are isomorphic they are not isometric with respect to 
standard metrics of ${\rm Gr}_2^+(\R^{n+2})$ and $S^n\times \R^{n+1}$.   
\end{remark}

Using the above basis of $T_P{\rm Gr}_2^+(\R^{n+2})$ 
the almost complex structure $J_1$ is characterized as follows. 
\begin{lemma}\label{characterization of J_1}
$J_1(f_{(e_3,0)})=\sqrt{1+\norm{v}^2}f_{(0, e_3)}$,  
$J_1(f_{(e_2, -\norm{v}e_1)})=\sqrt{1+\norm{v}^2}f_{(0, e_2)}$. 
\end{lemma}

Consider the standard symplectic structure $d\alpha$ on $V_0=TS^n$. 
It is easy to see that the frame 
$\{(e_3,0), \ (0, e_3), \ (e_2, -\norm{v}e_1), \ (0,e_2)\}$ 
form a symplectic basis of $T(TS^n)$. 
Lemma~\ref{characterization of J_1} shows that $J_1$ is 
compatible with $d\alpha$ and the following fact. 
\begin{prop}\label{J_0simJ_1}
Two almost complex structures $J_0$ and $J_1$ on $V_0$ are 
homotopic to each other in $d\alpha$-compatible almost complex structures.  
\end{prop}

On the other hand one can check the following by definition of 
the diffeomorphism in the proof of Lemma~\ref{X_r=Q_n}. 
\begin{lemma}\label{J_1=J_2}
$J_1=J_2$ on ${\rm Gr}_2^+(\R^{n+2})=Q_n$. 
\end{lemma}

\subsection{Localization and almost complex structures}
Hereafter we assume that $r$ is a positive integer $k$. 
Recall the diffeomorphism $X_k\cong X$ as in Lemma~\ref{X_r=X}. 
Since the diffeomorphism is equal to the identity map on the complement 
of the small $S^1$-invariant neighborhood $D_{r_{\varepsilon}}^o(\nu_{(k)})$ 
of $M_{(k)}$, 
Proposition~\ref{J_0simJ_1} 
guarantees that there exists an almost complex structure $J$ on $X_k$ 
which satisfies the following conditions. 
\begin{itemize}
\item $J$ is homotopic to $J_1$ on $X_k$. 
\item $J$ is $S^1$-invariant on $X_k\setminus D_{r_{\varepsilon}'}(\nu_{(k)})
\subset D_r^o(TS^n)$.
\item $J=J_0$ on $X_k\setminus D_{r_{\varepsilon}'}(\nu_{(k)})
=D_{r-2\varepsilon}^o(TS^n)$. 
\item $J=J_1$ on $D_{r_{\varepsilon}}^o(\nu_{(k)})$, 
\end{itemize}
where $r_\varepsilon=\sqrt{\frac{\varepsilon}{r-\varepsilon}}$ and $r_\varepsilon'=r_{2\varepsilon}=\sqrt{\frac{2\varepsilon}{r-2\varepsilon}}$. 

\begin{prop}\label{localization}
We have the following localization formula. 
$$
RR(X_k)=
RR_{\rm loc}(S^n)+
\sum_{l=1}^{k-1}RR_{\rm loc}(S_l(TS^n))+RR_{\rm loc}(Q_n). 
$$
\end{prop}
\begin{proof}
By the homotopy invariance of the Riemann-Roch number, 
we may use $J$ to compute $RR(X_k)$. 
On the other hand when we take 
$\varepsilon>0$ small enough so that there are no integers in 
the interval $[k-2\varepsilon, k)$, we may use $J$ to define 
local Riemann-Roch numbers 
in the right hand side by Lemma~\ref{BSorbits} and the 
properties of $J$. 
The equality follows from 
the excision formula of local Riemann-Roch numbers. 
\end{proof}

\section{Local models and their local Riemann-Roch number}
\subsection{Local models}
Let $M_0$ be a closed symplectic manifold. 
We assume that there exists a prequantizing line bundle 
$(L_0,\nabla_0)$ over $M_0$. 
Let $p_0:P_0\to M_0$ be a principal $S^1$-bundle over $M_0$ 
and $\alpha_0$ a principal connection of $P_0$, 
which is a pure imaginary 1-form on $P_0$. 
We define two open prequantized symplectic manifolds 
$M_D$ and  $M_{BS}$. 

\subsubsection{$M_{D}$}
Let $\C_1$ be the one-dimensional representation of $S^1$ of weight 1. 
Let $M_{\C}=P_0\times_{S^1}\C_1$ be the quotient space by the diagonal 
$S^1$-action. 
For a complex coordinate $z=x+\sqrt{-1}y$ of $\C_1$ and $r:=|z|$, 
the pure imaginary 1-form on $P_0\times \C_1$ 
$$
\tilde\alpha_{\C}:=\frac{1}{2}r^2\alpha_0+\frac{\sqrt{-1}}{2}(ydx-xdy)
$$is basic with respect to the $S^1$-action. 
Let $\alpha_{\C}$ be the 1-form on $M_{\C}$ whose pull-back 
to $P_0\times\C_1$ is equal to $\tilde\alpha_{\C}$. 
Define a Hermitian line bundle with connection 
$(L_{\C},\nabla_{\C})$ by 
$$
(L_{\C},\nabla_{\C}):=(p_{\C}^*L_{0},p_{\C}^*\nabla_0+\alpha_{\C}), 
$$where $p_{\C}:M_{\C}\to M_0$ is the projection. 

\begin{lemma}\label{modelMD}
Fix $0<\varepsilon <\sqrt{2}$ and we put 
$M_D:=P_0\times_{S^1}D^o_{\varepsilon}(\C_1)$. 
\begin{enumerate}
\item If $\varepsilon$ is small enough, 
then the restriction of $(L_{\C},\nabla_{\C})$ to $M_{D}$ 
is a prequantizing line bundle for a suitable symplectic structure on $M_D$. 
\item The natural $S^1$-action on $M_D$ using the $P_0$-component 
is Hamiltonian with the moment map $[u,z]\mapsto |z|^2/2$. 
\item The holonomy of the $S^1$-orbit through $[u,z]\in M_D$ 
is trivial if and only if $z=0$. 
\end{enumerate}
\end{lemma}

\begin{proof}
(1) The restriction of $(L_{\C},\nabla_{\C})$ to 
$M_0=P_0\times_{S^1}\{0\}$ is equal to $(L_0,\nabla_0)$, 
hence, the curvature 2-form of $\nabla_{\C}$ gives a 
symplectic structure for $\varepsilon>0$ small enough. 
By definition $(L_{\C}, \nabla_{\C})|_{M_D}$ is a prequantizing line bundle 
for this symplectic structure. 
(2) The $S^1$-action on the $P_0$-component preserves 
$\nabla_{\C}$, and the evaluation of the infinitesimal action 
by $\alpha_{\C}$ is equal to $|z|^2/2$. 
(3) follows from (2). 
\end{proof}

\subsubsection{$M_{BS}$}
Define a pure imaginary 1-form $\alpha_{BS}$ on $P_0\times \R$ by 
$$
\alpha_{BS}=r\alpha_0, 
$$where $r$ is the coordinate of $\R$. 
Define a Hermitian line bundles with connection 
over $P_0\times\R$ by 
$$
(L_{BS}, \nabla_{BS}):=
(p_0^*L_0,p^*_0\nabla_0+\alpha_{BS}). 
$$

By the similar argument as in the proof of Lemma~\ref{modelMD} 
we have the following. 
\begin{lemma}Fix $0<\varepsilon<1$ and we put 
$M_{BS}=P_0\times(-\varepsilon,\varepsilon)$. 
\begin{enumerate}
\item If $\varepsilon$ is small enough, 
then the restriction of $(L_{BS},\nabla_{BS})$ to 
$M_{BS}$ is a prequantizing line bundle 
for a suitable symplectic structure on $M_{BS}$. 
\item The natural $S^1$-action on the $P_0$-component 
is Hamiltonian with the moment map $(u,r)\mapsto r$. 
\item The holonomy of the $S^1$-orbit through $(u,r)\in M_{BS}$ 
is trivial if and only if $r=0$. 
\end{enumerate}
\end{lemma}


\subsection{Local Riemann-Roch numbers}
Using the free $S^1$-action on $M_{D}\setminus M_0$ and 
$M_{BS}\setminus P_{0}$ 
we can define their local Riemann-Roch number 
$RR_{\rm loc}(M_0)$ and $RR_{\rm loc}(P_{0})$. 
On the other hand since $M_0$ is closed 
the usual Riemann-Roch number $RR(M_0)$ is defined. 

\begin{lemma}\label{productformula}
We have the following equalities
$$
RR_{\rm loc}(M_0)=RR_{\rm loc}(P_{0})=
RR(M_0). 
$$
\end{lemma}

\begin{proof}
These equalities follows from 
the product formula and 
the facts 
$RR_{\rm loc}(D_{\varepsilon}^o)=
RR_{\rm loc}(S^1\times (-\varepsilon,\varepsilon))=1$ as 
$S^1$-equivariant local Riemann-Roch numbers. 
Note that we use the identification 
$P_0\times (-\varepsilon, \varepsilon)=
P_0\times_{S^1}(S^1\times(-\varepsilon,\varepsilon))$. 
\end{proof}

\subsection{A formula for local Riemann-Roch numbers}
Let $(M,L)$ be a prequantized symplectic manifold 
with a Hamiltonian $S^1$-action and its moment map $\mu:M\to \R$. 
We assume that $0$ is a regular value of $\mu$ and 
$\mu^{-1}(0)$ is a compact submanifold of $M$. 
We also assume that the $S^1$-action on $\mu^{-1}(0)$ is free. 
Let $M_0:=\mu^{-1}(0)/S^1$ be the symplectic quotient at $0$, 
which is a compact symplectic manifold with a 
prequantizing line bundle $L_0:=L|_{\mu^{-1}(0)}/S^1$. 
We put $P_0:=\mu^{-1}(0)$ and then the natural projection 
$P_0\to M_0$ gives a structure of a principal $S^1$-bundle over $M_0$. 
Consider the symplectic cut of $M$ at 0,  
$$
M_{\rm cut}:=\mu^{-1}(-\infty, 0)\cup M_{0}. 
$$
For these data we have two local Riemann-Roch numbers 
$RR_{\rm loc}(P_0)$ and $RR_{\rm loc}(M_{0})$.  

\begin{theorem}\label{formulaforRR}
We have the following equalities
$$
RR_{\rm loc}(P_0)=RR_{\rm loc}(M_{0})
=RR(M_0). 
$$
\end{theorem}
\begin{proof}
A version of Darboux's theorem (\cite[Proposition~5.11]{Fujita-Furuta-Yoshida3}) implies that 
we can use local models $M_{BS}$ and $M_D$ for 
$M_0$ and $P_0$. 
The equalities follow from 
the product formula (Lemma~\ref{productformula}). 
\end{proof}

\section{Proof of the main theorem}
In this section we use following notations 
for a fixed positive integer $k$ and an integer $l$ with $0<l\le k$. 
\begin{itemize}
\item $X_k=D_k^o(TS^n)\cup h^{-1}_{01}(k)/S^1
\cong{\rm Gr}_2^+(\R^{n+2})\cong Q_n$
\item $P_{(l)}=h^{-1}_{01}(l)=S_l(TS^n)$
\item $M_{(l)}=P_{(l)}/S^1\cong Q_{n-1}$
\item $\displaystyle RR_n^l:=\binom{n+l}{n}+\binom{n+l-1}{n}$, 
where $\displaystyle\binom{\cdot}{\cdot}$ is the binomial coefficient. 
\end{itemize}

\subsection{Pre-quantizing line bundles}
Let $(L_{0},d-2\pi\sqrt{-1}\alpha)$ be the 
prequantizing line bundle over $V_0$, 
where $L_{0}$ is the trivial line bundle and $\alpha$ is 
the Liouville 1-form. 
Note that the weight  of the $S^1$-action 
on the space of global parallel sections 
of $L_{0}|_{P_{(l)}}$ is equal to $l$, 
and hence, the prequantizing line bundle 
over $M_{(l)}$ is given by $P_{(l)}\times_{S^1}\C_{l}$, 
where $\C_l$ is the complex line with the standard 
$S^1$-action of weight $l$. 
On the other hand, as in 
Remark~\ref{normalbundle} and Remark~\ref{normalbundle2}, 
$P_{(l)}$ is isomorphic to the unit circle bundle of 
the tautological plane bundle $\gamma_2\cong {\mathcal O}(1)|_{Q_{n-1}}
\to {\rm Gr}_2^+(\R^{n+1})\cong Q_{n-1}$. 
Summarizing we have the following. 

\begin{lemma}
Under the identification $M_{(l)}\cong Q_{n-1}$, 
$L_{(l)}:={\mathcal O}(l)|_{Q_{n-1}}$ gives a prequantizing line bundle 
over $M_{(l)}$. 
\end{lemma}

On the other hand, the symplectic cutting construction 
for $(V_0, L_{0})$ yields a prequantizing line bundle 
$L_k\to X_k$ whose restriction to the quotient $M_{(k)}$ is 
given by $(L_{0}|_{h^{-1}_{01}(k)}\otimes \C_k)/S^1=L_{(k)}$.  

\begin{lemma}
Under the identification $X_k\cong Q_n$, 
the pre-quantizing line bundle $L_k$ is isomorphic 
to ${\mathcal O}(k)$ as a prequantizing line bundle. 
\end{lemma}
\begin{proof}
Note that since $Q_n$ is simply connected the isomorphism class 
of pre-quantizing line bundles is unique. 
The required isomorphism follows from the isomorphism 
$H^2(Q_n,\Z)\stackrel{\cong}{\rightarrow} H^2(Q_{n-1},\Z)$,  
$c_1({\mathcal O}(k))\mapsto c_1(L_{(k)})$. 
\end{proof}

\subsection{Riemann-Roch number of the compactification} 
\begin{lemma}\label{dim O(l)}
For each positive integer $l$, 
the dimension of the space of holomorphic sections 
of ${\mathcal O}(l)|_{Q_n}$ is equal to $RR_n^l$. 
\end{lemma}
\begin{proof}
The involution 
$$
[z_0:z_1:\ldots:z_n:z_{n+1}]\mapsto 
[z_0:z_1:\ldots:z_n:-z_{n+1}]
$$acts on $Q_n$ and its orbit space is $\CP^{n}$. 
The projection map $Q_n\to \CP^n$ given by 
$[z_0:\ldots:z_n:z_{n+1}]\mapsto [z_0:\ldots:z_n]$ is 
a branched covering with branching locus $Q_{n-1}\subset\CP^n$. 
The involution lifts to ${\mathcal O}(l)|_{Q_n}$ so that 
the quotient bundle is ${\mathcal O}(l)\to \CP^n$. 
Consider the decomposition of the space of holomorphic sections 
with respect to the involution, 
$H^0(Q_n,{\mathcal O}(l))=H^+_{n,l}\oplus H^-_{n,l}$, 
where the involution acts on $H^{\pm}_{n,l}$ by $\pm 1$. 
The invariant part $H^+_{n,l}$ is isomorphic to 
$H^0(\CP^{n},{\mathcal O}(l))$, 
which is the vector space of
homogeneous polynomial of degree $l$ of 
$(n+1)$-variable $z_0,\ldots,z_{n}$. 
Its dimension is given by the binomial coefficient 
$\binom{n+l}{n}$. 
On the other hand any section $s\in H^-_{n,l}$ can be divisible by $z_{n+1}$, 
and $s/z_{n+1}$ defines a section in $H_{n,l-1}^+$. 
In particular we have 
$\dim H_{n,l}^-=\dim H^0(\CP^{n},{\mathcal O}(l-1))=
\binom{n+l-1}{n}$. 
\end{proof}

\begin{prop}\label{RRofXk}
The Riemann-Roch number of $X_k$ with the prequantizing line bundle 
$L_{k}$ is given by $RR(X_k)=RR_n^k$. 
\end{prop}
\begin{proof}
We use the identifications 
$X_k=Q_n$, $L_k={\mathcal O}(k)$ and 
the natural complex structure $J_2=J_1$ of $Q_n\subset \CP^{n+1}$.  
By the Kodaira vanishing theorem, $RR(X_k)=RR(Q_n)$ is equal to the 
dimension of the space of holomorphic sections 
$H^0(Q_n,{\mathcal O}(k))=RR^k_n$. 
\end{proof}

\subsection{Localization}
Recall that the complement of $S^n$ and 
$h_{01}^{-1}(k)/S^1\cong Q_{n-1}$ in $X_k$ has a free Hamiltonian $S^1$-action 
induced by the geodesic flow. 
The holonomy representation of the 
restriction of $L_k$ to a orbit is trivial if and only if 
the orbit is contained in $P_{(l)}=h^{-1}_{01}(l)$ for some integer $l$. 

\begin{proof}[Proof of Theorem~\ref{RRofBS}]
The required formula follows from 
Theorem~\ref{formulaforRR}, Lemma~\ref{dim O(l)} 
and the identification $M_{(l)}\cong Q_{n-1}$. 
\end{proof}

We also have the 
local Riemann-Roch number $RR_{\rm loc}(M_{(k)})$ in $X_k$ 
with the prequantizing line bundle $L_k|_{M_{(k)}}$. 
\begin{prop}\label{RRofMk}
$RR_{\rm loc}(M_{(k)})=RR_{n-1}^k$. 
\end{prop}
\begin{proof}
The required equality follows from the identifications 
$M_{(k)}\cong Q_{n-1}$, $L_{k}|_{M_{(k)}}\cong {\mathcal O}(k)|_{Q_{n-1}}$, 
Theorem~\ref{formulaforRR} and Lemma~\ref{dim O(l)}. 
\end{proof}

\begin{lemma}\label{binom}
For $n\ge 1$, we have the following equality. 
$$
RR_{n-1}^0+RR_{n-1}^1+\cdots+RR_{n-1}^k=RR_n^k. 
$$
\end{lemma}
\begin{proof}
Compare the coefficients of $a^{n-1}$ in the equality 
$$
(a+1)^{n-1}+(a+1)^{n}+\cdots+(a+1)^{n+k-1}=
(a+1)^{n-1}\left((a+1)^{k+1}-1\right)/a. 
$$
\end{proof}

\begin{proof}[Proof of Theorem~\ref{RRofSn}]
By Proposition~\ref{localization} we have 
$$
RR_{\rm loc}(S^n)+
\sum_{l=1}^{k-1}RR_{\rm loc}(P_{(l)})+RR_{\rm loc}(M_{(k)})=RR(X_k),  
$$and hence, 
$RR_{\rm loc}(S^n)=RR_{n-1}^0$ for $n\ge 1$ by 
Theorem~\ref{RRofXk}, Theorem~\ref{RRofBS}, 
Proposition~\ref{RRofMk} and Lemma~\ref{binom}. 
For $n=0$ we have $RR_{\rm loc}(S^0)=2$ by definition. 
\end{proof}


\bibliographystyle{amsplain}
\bibliography{reference}
\end{document}